\documentclass[graybox]{svmult}
\usepackage{type1cm}        
\usepackage{makeidx}        
\usepackage{graphicx}               
\usepackage{multicol}        
\usepackage[bottom]{footmisc}
\usepackage{newtxtext}       
\usepackage{newtxmath} 
\newcommand{\RN}[1]{
  \textup{\uppercase\expandafter{\romannumeral#1}}
}

\makeindex       

\begin{document}

\title*{On the connection between the Goldbach conjecture and the Elliott-Halberstam conjecture}
\titlerunning{Goldbach conjecture and Elliott-Halberstam conjecture}

\author{Jing-Jing Huang and Huixi Li}

\institute{Jing-Jing Huang \at Department of Mathematics and Statistics, University of Nevada, Reno, 1664 N Virginia St, Reno, NV 89557 \\ \email{jingjingh@unr.edu}
\and Huixi Li \at Department of Mathematics and Statistics, University of Nevada, Reno, 1664 N Virginia St, Reno, NV 89557 \\ \email{huixil@unr.edu}}

\maketitle

\hfill \textit{Dedicated to Professor Melvyn Nathanson on the occasion of his 75th birthday.} 

\ \\

\abstract{\,\,In this paper we prove that the binary Goldbach conjecture for sufficiently large even integers would follow under the assumption that both the Elliott-Halberstam conjecture and a variant of the Elliott-Halberstam conjecture twisted by the M\"{o}bius function hold, in which the sum of their levels of distribution exceeds~1. This continues the work of Pan in 1982. An analogous result for the twin prime conjecture has been obtained by Ram Murty and Vatwani in 2017. } 

\ \\


\section{Introduction}
The well-known Chen's theorem on the Goldbach conjecture states that every sufficiently large even integer $N$ can be written as the sum of two positive integers $p$ and~$q$, where $p$ is a prime and $q$ is an almost prime with at most two prime divisors~\cite{Chen}. The barrier from Chen's theorem to the Goldbach conjecture has been well known as the parity problem in sieve theory: one can not tell whether $q = N - p$ has exactly one or two prime divisors. 

Let $r(N)$ be the number of representations of $N$ as the sum of two primes, and  let $\widetilde{r}(N)$ denote the weighted number of representations
\[
\widetilde{r}(N)=\sum_{n < N} \Lambda(n) \Lambda(N - n),
\]
where the von Mangoldt function $\Lambda(n)$ is defined as
\[
\Lambda(n) = 
\begin{cases}
\log p, & \text{if } n = p^\ell \text{ for some prime $p$ and integer $\ell \geq 1$}, \\  
0, & \text{otherwise. }
\end{cases}
\]
It is easy to see~\cite[Page 555]{Pan} that  
\begin{equation*}
r(N) = \frac{\widetilde{r}(N)}{(\log N)^2}
\left(1 + O \left( \frac{\log \log N}{\log N} \right) \right) + O \left( \frac{N}{(\log N)^3} \right).
\end{equation*}
Since it is somewhat more convenient to deal with the smoothed sum $\widetilde{r}(N)$ rather than $r(N)$ itself, we will only focus on the study of the former in this paper. 

By a heuristic argument based on the circle method, Hardy and Littlewood~\cite{HardyLittlewood} have conjectured that
\begin{equation}\label{asymp}
\widetilde{r}(N)\sim \mathfrak{S}(N)N,
\end{equation}
where
\begin{equation}\label{frakSN}
\mathfrak{S}(N) 
=
\begin{cases}
2 \prod_{p > 2} \left( 1 - \frac{1}{(p - 1)^2} \right)
\prod_{\substack{p \mid N \\ p > 2}} \left( 1 + \frac{1}{p - 2} \right), & \text{if $N$ is even, } \\
0, & \text{if $N$ is odd. }
\end{cases}
\end{equation}
Pan~\cite{Pan} has made a new attempt on the Goldbach conjecture, which gives an alternative way to suggest the expected asymptotic formula~\eqref{asymp}. In some sense, Pan's approach is more direct and elementary than the Hardy-Littlewood circle method. Pan's main result~\cite[Theorem \RN{1}, Theorem \RN{2}]{Pan} states that for sufficiently large even integers $N$ and for $Q = N^{\frac{1}{2}} (\log N)^{-20}$,
we have
\[
\widetilde{r}(N)
= R_0(N)
+
R_1(N) + R_2(N) + R_3(N) + O\left( \frac{N}{\log N} \right), 
\]
where
\begin{equation}\label{R0}
R_0(N) = \sum_{n < N} 
\left( 
\sum_{\substack{d_1 \mid n \\ d_1 \leq Q}} \mu(d_1) \log d_1
\right) 
\left(
\sum_{\substack{d_2 \mid (N - n) \\ (d_2, N) = 1 \\ d_2 \leq Q}} \mu(d_2) \log d_2
\right)
= \mathfrak{S}(N) N + O\left( \frac{N}{\log N} \right), 
\end{equation}
\begin{equation}\label{R1}
R_1(N) = \sum_{n < N} 
\left( 
\sum_{\substack{d_1 \mid n \\ d_1 \leq Q}} \mu(d_1) \log d_1
\right) 
\left(
\sum_{\substack{d_2 \mid (N - n) \\ (d_2, N) = 1 \\ d_2 > Q}} \mu(d_2) \log d_2
\right)
, 
\end{equation}
\begin{equation}\label{R2}
R_2(N) = \sum_{n < N} 
\left( 
\sum_{\substack{d_1 \mid n \\ d_1 > Q}} \mu(d_1) \log d_1
\right) 
\left(
\sum_{\substack{d_2 \mid (N - n) \\ (d_2, N) = 1 \\ d_2 \leq Q}} \mu(d_2) \log d_2
\right), 
\end{equation}
and 
\begin{equation}\label{R3}
R_3(N) = \sum_{n < N} 
\left( 
\sum_{\substack{d_1 \mid n \\ d_1 > Q}} \mu(d_1) \log d_1
\right) 
\left(
\sum_{\substack{d_2 \mid (N - n) \\ (d_2, N) = 1 \\ d_2 > Q}} \mu(d_2) \log d_2
\right). 
\end{equation}
Furthermore, by the Bombieri-Vinogradov theorem, we have
\[
R_i(N)= O\left( \frac{N}{\log N} \right),\quad i=1, 2. 
\]

As Pan points out, the difficulty in the Goldbach conjecture  lies in the study of $R_3(N)$. Moreover, the Hardy-Littlewood conjecture reduces to showing that $R_3(N)$ is  inferior to the expected main term $\mathfrak{S}(N) N$. Our main purpose of this paper is to show that the Goldbach conjecture for sufficiently large even integers follows under the assumptions that the Elliott-Halberstam conjecture holds with level of distribution $\theta$ and that a variant of the Elliott-Halberstam conjecture twisted by the M\"{o}bius function holds with level of distribution $\theta'$, where $\theta+\theta'>1$.

An analogous result in the twin prime conjecture setting has been obtained by Ram Murty and Vatwani~\cite{MuV}. This perhaps is not surprising since we know the Goldbach conjecture and the twin prime conjecture are closely connected. For example, Chen has also proved in~\cite{Chen} that for every even integer $h \neq 0$, there are infinitely many pairs of integers $(p, q)$, where $p$ is a prime and $q = p + h$ is an almost prime with at most two prime divisors. Again, the barrier from this theorem to the twin prime conjecture is caused by the parity problem. There have been many attempts to break the parity barrier in different settings~\cite{FI, FrIw, FrIw2, Pintz, MuV}. 

Ram Murty and Vatwani formulate in~\cite{MuV} a conjecture regarding the equidistribution of the M\"{o}bius function over shifted primes in arithmetic progressions (cf.~\cite[(1.4)]{MuV}), and they prove that such a conjecture for a fixed even integer $h \neq 0$, in conjunction with the Elliott-Halberstam conjecture, can remove the parity barrier and produce infinitely many pairs of primes $(p, p + h)$. Note that if we assume the Elliott-Halberstam conjecture holds, i.e., the conjecture $EH(N^\theta)$ introduced below holds for all $\theta < 1$, 
then it is proved in~\cite{Maynard} that there are infinitely many pairs of consecutive primes that differ by at most 12. Moreover, the bound has been improved in~\cite{Polymath} to 6  assuming a generalized Elliott-Halberstam conjecture (cf.~\cite[Claim~2.6]{Polymath}). Unconditionally, after the sensational breakthrough of Yitang Zhang~\cite{Zhang} and subsequent work of Maynard~\cite{Maynard}, the best known bound on small gaps between primes has been reduced to  246 by the Polymath group~\cite{Polymath}.  

To put the main result in perspective, we state the Elliott-Halberstam conjecture and a variant of it twisted by the M\"{o}bius function. 

\noindent \textbf{The Elliott-Halberstam conjecture $EH(N^\theta(\log N)^C)$}. 
 For any $A > 0$, we have
\[
\sum_{q \leq N^\theta(\log N)^C} 
\max_{y \leq N}
\max_{\substack{(a, q) = 1}}
\left| 
\sum_{\substack{n \leq y \\ n \equiv a \pmod{q}}} \Lambda(n) - \frac{y}{\phi(q)}
\right| 
\ll \frac{N}{(\log N)^A}. 
\]

\noindent \textbf{A variant of the Elliott-Halberstam conjecture twisted by the M\"{o}bius function $EH_{\mu}(N^\theta)$}.  For any $A > 0$, we have
\small
\[
\sum_{q \leq N^\theta} 
\max_{y < N}
\max_{\substack{(a, q) = 1}}
\left| 
\sum_{\substack{n \leq y \\ n \equiv a \pmod{q}}} \Lambda(n) \mu(N - n)
- \frac{1}{\phi(q)} \sum_{n \leq y} \Lambda(n) \mu(N - n)
\right| 
\ll \frac{N}{(\log N)^A}. 
\]
\normalsize

Now we are poised to state our main theorem.
\begin{theorem}\label{maintheorem}
For a given $A>0$, suppose that $EH(N^\theta (\log N)^{2A+8})$ and $EH_{\mu}(N^{1 - \theta})$ are true for some constant $0 < \theta < 1$. Then for all positive even integers $N$, we have
\[
\widetilde{r}(N)
\geq 
\mathfrak{S}(N) (1 - \mathcal{A}(N)) N
+ O\left( \frac{N}{(\log N)^A} \right), 
\]
where 
\begin{equation}\label{Ah}
\mathcal{A}(N) = \prod_{\substack{p \nmid N \\ p > 2}}
\left( 1 - \frac{1}{p(p - 1)} \right)
\end{equation}
and the implicit constant depends on $A$ and $\theta$.
Moreover, the assertions $\widetilde{r}(N) \sim \mathfrak{S}(N) N$ and $\sum_{n < N} \Lambda(n) \mu(N - n) = o(N)$ are equivalent. \end{theorem}
The following corollary is an immediate consequence of our main theorem.
\begin{corollary}\label{coro}
Suppose that $EH(N^\theta (\log N)^{12+\epsilon})$ and $EH_{\mu}(N^{1 - \theta})$ are true for some constants $0 < \theta < 1$ and $\epsilon>0$. Then the Goldbach conjecture holds for all sufficiently large even integers, i.e., every sufficiently large even integer can be written as a sum of two prime numbers. In particular, in view of the Bombieri-Vinogradov theorem, the above conclusion holds if the conjecture $EH_\mu(N^{\theta'})$ is true for some $\theta' > \frac{1}{2}$.
\end{corollary}

Our work can be naturally regarded as a continuation of Pan~\cite{Pan} for the following reasons. While Pan truncates the sum over $d_1$ and $d_2$ at $N^{\frac12}(\log N)^{-20}$ due to the limitation of the Bombieri-Vinogradov theorem, it is however not necessary to do this as long as one can handle all the truncated sums properly. Moreover, Pan leaves the sum $R_3(N)$ untouched and in particular does not give any hint on how to approach it. The main thrust of the current memoir is to show that we can estimate $R_3(N)$ if we have some knowledge about the equidistribution of $\Lambda(n)\mu(N-n)$ in arithmetic progressions. 

As in~\cite{MuV}, we remark that the proof goes through if we only assume equidistribution in $EH(N^\theta (\log N)^{2A + 8})$ and $EH_\mu(N^{1 - \theta})$ for the fixed residue class $n \equiv N \pmod{q}$ instead of taking the maximum of all residue classes coprime to $q$. Also by assuming other variants of the Elliott-Halberstam conjecture, our argument should give a lower bound on the number of representations of a large integer $N$ as a linear combination $a p + b q$ of primes $p$ and $q$, where $a$ and $b$ are fixed positive integers, provided that there is no local obstruction.

The interested readers are referred to~\cite{Vatwani} for other types of Elliott-Halberstam conjectures and their relations to the twin prime conjecture. It remains to be seen whether the conjectures mentioned in~\cite{Vatwani} are related to the Goldbach conjecture. Finally, it is worth noting that Hua has proposed another elementary approach to the Goldbach conjecture~\cite{Hua}. 

In Section~\ref{lemmasection} we introduce some technical lemmas that will be applied later.
In Section~\ref{proofofmaintheorem} we prove Theorem~\ref{maintheorem} and Corollary~\ref{coro} after we first evaluate two important sums assuming the conjectures $EH(N^\theta (\log N)^{2A + 8})$ and $EH_\mu(N^{1 - \theta})$ hold. Throughout the paper, we will use  Vinogradov's symbol $f(x)\ll g(x)$ and Landau's symbol $f(x)=O(g(x))$ to mean there exists a constant $C$ such that $|f(x)|\le Cg(x)$. We use $\epsilon$ to denote any sufficiently small positive number, which may not necessarily be the same in each occurrence. 
\section{Preliminary Lemmata}\label{lemmasection}
Our first lemma is a result of Goldston and Y\i ld\i r\i m involving the singular series $\mathfrak{S}(N)$ defined in~\eqref{frakSN}.

\begin{lemma}[{\cite[Lemma 2.1]{GY}}] For any positive integer $N$ and any real number $R$ satisfying $\log N\ll\log R$, there exists an absolute positive constant $c_1$ such that 
\[
\sum_{\substack{d \leq R \\ (d, N) = 1}}
\frac{\mu(d)}{\phi(d)}
\log \left(\frac{R}{d}\right)
= 
\mathfrak{S}(N)
+ O\left( e^{-c_1 \sqrt{\log R}} \right)
\]
and 
\[
\sum_{\substack{d \leq R \\ (d, N) = 1}}
\frac{\mu(d)}{\phi(d)}
= 
O\left( e^{-c_1 \sqrt{\log R}} \right). 
\]
\label{GYlemma}
\end{lemma}

Our next lemma has similar flavor as Lemma~\ref{GYlemma}. It can be viewed as a quantitative refinement of~\cite[Proposition 3.2, Proposition 3.3]{MuV} by Ram Murty and Vatwani. 

\begin{lemma}\label{serieslemma}
For any positive even integer $N$ and any real number $R$ satisfying $\log N\ll\log R$, there exists some constant $c_2 > 0$ such that 
\begin{equation}\label{series}
\mathcal{A}(N)
\sum_{\substack{d \le R \\ (d, N) = 1}} 
\frac{\mu(d) \log(1/d) g_N(d)}{\phi(d)} 
= \mathfrak{S}(N)+
O\left(e^{-c_2 \sqrt{\log R}} \right),
\end{equation}
where 
\begin{equation}\label{gNd}
g_N(d) = \prod_{\substack{p \mid d \\ p \nmid N}}  \frac{(p - 1)^2}{p^2 - p - 1}. 
\end{equation}
\end{lemma}

\begin{proof}
Let 
\[
f_N(s)=\sum_{\substack{d=1 \\ (d, N) = 1}} ^\infty
\frac{\mu(d)g_N(d)}{\phi(d)d^{s}}.
\]
Noting that the series on the right hand side admits the Euler product
\[
f_N(s)=\prod_{p\nmid N}\left(1-\frac{g_N(p)}{(p-1)p^s}\right)=\prod_{p\nmid N}\left(1-\frac{p-1}{(p^2-p-1)p^s}\right),
\]
we immediately see that $f_N(s)$ converges absolutely for $\Re(s)>0$ and therefore defines an analytic function there. Actually we may write 
\[
f_N(s)=\zeta(s+1)^{-1}h_N(s), 
\]
where
\[
h_N(s)=\prod_{p\mid N}\left(1-\frac1{p^{s+1}}\right)^{-1}\prod_{p\nmid N}
\left(\left(1-\frac1{p^{s+1}}\right)^{-1}\left(1-\frac{p-1}{(p^2-p-1)p^s}\right) \right).
\]
Then, it is readily verified that $h_N(s)$ is analytic for $\Re(s)>-1$. So $f_N(s)$ can be meromorphically continued to $\Re(s)>-1$ with only poles at zeros of $\zeta(s+1)$.

Hence, 
\[
f_N'(s)=\sum_{\substack{d=1 \\ (d, N) = 1}} ^\infty
\frac{\mu(d)\log(1/d)g_N(d)}{\phi(d)d^{s}}=\frac{h_N'(s)}{\zeta(s+1)}-h_N(s)\frac{\zeta'(s+1)}{\zeta(s+1)^2}
\]
is also a meromorphic function for $\Re(s)>-1$ with only poles at zeros of $\zeta(s+1)$. Let 
\[
a_N(d)=\begin{cases}
\frac{\mu(d)\log(1/d)g_N(d)}{\phi(d)},& \textrm{if }(d,N)=1,\\
0,& \textrm{otherwise}.
\end{cases}
\]
Thus noting that $g_N(d)\ll1$ and $\phi(d)\gg d/\log \log d$ by~\cite[Theorem \RN{1}.5.6]{Ten}, we have
\[
|a_N(d)|\ll \frac{(\log d)^2}d.
\]
Here we adopt the conventional notation $s=\sigma+it$.
By Perron's formula (\cite[Corollary 5.3]{MV}), we have
\[
\sum_{d\leq R}a_N(d)=\frac{1}{2\pi i}\int_{\sigma_0-iT}^{\sigma_0+iT}f_N'(s)\frac{R^s}{s}\,ds+E_N,
\]
where $\sigma_0>0$ and 
\[ 
E_N\ll\sum_{\frac{R}{2}<d<2R}|a_N(d)|\min\left(1,\frac{R}{T|d-R|}\right)+
\frac{4^{\sigma_0}+R^{\sigma_0}}{T}\sum_{d=1}^{\infty}\frac{|a_N(d)|}{d^{\sigma_0}}.
\]
Suppose that $2\le T\le R$ and that $\sigma_0=\frac1{\log R}$. Then we obtain via some elementary calculation that
\[
E_N\ll \frac{(\log R)^3}{T}.
\]

Let $\sigma_1=-c_2/\log T$, where $c_2$ a small positive constant to be chosen later.
The integrand $f_N'(s)\frac{R^s}{s}$ is a meromorphic function for $\Re(s)>-1$ with a simple pole at $s=0$.  It follows from the residue theorem that
\begin{align*}
&\frac{1}{2\pi i}
\int_{\sigma_0-iT}^{\sigma_0+iT}
f_N'(s)\frac{R^s}{s} \, ds \\
&= \text{Res}_{s=0}
\left(f_N'(s)\frac{R^s}{s}\right)+
\frac{1}{2\pi i}
\left(\int_{\sigma_0-iT}^{\sigma_1-iT}+\int_{\sigma_1-iT}^{\sigma_1+iT}+\int_{\sigma_1+iT}^{\sigma_0+iT}\right)
f_N'(s)\frac{R^s}{s} \, ds.
\end{align*}
First of all, it is easy to check that for positive even integers $N$ we have
\[
\text{Res}_{s=0}\left(f_N'(s)\frac{R^s}{s}\right)=f_N'(0)=h_N(0)= \frac{\mathfrak{S}(N)}{\mathcal{A}(N)}.
\]
Recalling the classical zero-free region of the zeta function, we will choose $c_2$ small enough such that $\zeta(s)$ is nonvanishing when
\[
\sigma\ge1-\frac{c_2}{\log |t|}, \quad |t|\ge 2, 
\]
and moreover, in this region we have
\[
\frac1{\zeta(s)}\ll \log|t|,
\quad
\frac{\zeta'(s)}{\zeta(s)}\ll \log |t|.
\]
Also when $\sigma\ge-\frac14$, we have
\small
\begin{align*}
\log|h_N(s)|\ll -\sum_{p|N}\log\left(1-\frac1{p^{\sigma+1}}\right)
\ll \sum_{p|N}\frac1{p^{\sigma+1}}
\ll \sum_{p<2\log(2N)}\frac1{p^{\sigma+1}}
\ll(\log N)^{\frac14},
\end{align*}
\normalsize
therefore
\[
|h_N(s)|\ll e^{\sqrt[4]{\log N}}
\]
and
\begin{align*}
    |h_N'(s)|&\ll e^{\sqrt[4]{\log N}}\sum_{p|N}\frac{\log p}{p^{\sigma+1}}
    \ll e^{\sqrt[4]{\log N}} \log N.
\end{align*}
So when $\sigma\ge\sigma_1$ and $|t|\le T$, we have
\[
|f_N'(s)|\ll e^{\sqrt[4]{\log N}} (\log N + \log T)^2\ll e^{c_3\sqrt[4]{\log R}}
\]
in view of the assumption that $\log N\ll \log R$.
Then
\[
\int_{\sigma_1+iT}^{\sigma_0+iT}
f_N'(s)\frac{R^s}{s} \, ds
\ll e^{c_3\sqrt[4]{\log R}}\frac{R^{\sigma_0}}T(\sigma_0-\sigma_1)
\ll \frac{e^{c_3\sqrt[4]{\log R}}}T.
\]
Similarly, we may treat the integral from $\sigma_0-iT$ to $\sigma_1-iT$. Lastly, we observe that
\begin{align*}
&\int_{\sigma_1-iT}^{\sigma_1+iT}
f_N'(s)
\frac{R^s}{s} \, ds \\
&\ll 
e^{c_3\sqrt[4]{\log R}}
R^{\sigma_1}
\left(\int_{-T}^T\frac{1}{|t|+1} \, dt
+\int_{-1}^1\frac1{|\sigma_1-it|} \, dt\right)\\
&\ll e^{c_3\sqrt[4]{\log R}}R^{\sigma_1}\left(\log T+\sigma_1^{-1}\right)\\
&\ll e^{c_3\sqrt[4]{\log R}} R^{\sigma_1}\log T.
\end{align*}
Combining the above estimates, we choose $T=e^{\sqrt{c_2\log R}}$ and hitherto obtain
\begin{align*}
\sum_{d\leq R}a_N(d)
=&\frac{\mathfrak{S}(N)}{\mathcal{A}(N)}
+O\left((\log R)^3\left(\frac1T+R^{-\frac{c_2}{\log T}}e^{c_3\sqrt[4]{\log R}}\right)\right) \\
=& \frac{\mathfrak{S}(N)}{\mathcal{A}(N)}
+ O \left(
(\log R)^3 e^{-\sqrt{c_2\log R}+c_3\sqrt[4]{\log R}} 
\right) \\
=& \frac{\mathfrak{S}(N)}{\mathcal{A}(N)}
+ O \left(
e^{-c_2\sqrt{\log R}}
\right),
\end{align*}
as we may take $c_2<1$.
Since $\mathcal{A}(N)\le 1$, this completes the proof of the lemma. 
\end{proof}

The following two lemmas are variations of some results in~\cite{MuV}. They will be applied in Section~\ref{S1section} and Section~\ref{S2section}, respectively.
\begin{lemma}\label{S11}
Let $A > 0$ and $0 < \theta < 1$ be two positive real numbers. Then 
\[
\sum_{d \leq N^\theta} \mu(d)  \log(1/d)
\sum_{\substack{b \leq B \\ (b d, N) = 1}}
\frac{\mu(b)}{\phi([b^2, d])}
= \mathfrak{S}(N)
+ O \left( \frac{1}{(\log N)^A} \right)
\]
holds for all positive even integers $N$, where $B = (\log N)^{A + 4}$.
\end{lemma}
\begin{proof}
The proof is similar to the proof of the asymptotic formula for 
\[
S_{11}(\beta, x, h) = 
\sum_{d \leq \beta} \mu(d)  \log(1/d)
\sum_{\substack{b \leq B \\ (b d, h) = 1}}
\frac{\mu(b)}{\phi([b^2, d])}
\]
defined on~\cite[Page 653]{MuV}, where $x$ is a large real number, $\beta = x^{\theta''}$ for some $0 < \theta'' < 1$, and $h \neq 0$ is a fixed even integer. Our case here is trickier, since as shown later in the proof, we need to worry about the rate of convergence of the series in~\eqref{series}, which is taken care of by Lemma~\ref{serieslemma}. 

For square-free integers $b \leq B$ and $d \leq N^{\theta}$, we have $\left(b^2, \frac{d}{(b, d)} \right) = 1$ and $\left((b, d), \frac{d}{(b, d)}\right) = 1$, and thus 
\[
\phi([b^2, d]) = \phi\left(b^2 \frac{d}{(b, d)}\right) = \frac{b \phi(b) \phi(d)}{ \phi((b, d))}.
\] 
So
\small
\[
\sum_{d \leq N^\theta} \mu(d)  \log(1/d)
\sum_{\substack{b \leq B \\ (b d, N) = 1}}
\frac{\mu(b)}{\phi([b^2, d])}
= 
\sum_{d \leq N^\theta} 
\frac{\mu(d)  \log(1/d)}{\phi(d)}
\sum_{\substack{b \leq B \\ (b d, N) = 1}}
\frac{\mu(b) \phi((b, d))}{b \phi(b)}. 
\]
\normalsize
We extend the sum over $b$ to infinity with an error
\small
\begin{align*}
\sum_{d \leq N^\theta} 
\frac{\mu(d)  \log(1/d)}{\phi(d)}
\sum_{\substack{b > B \\ (b d, N) = 1}}
\frac{\mu(b) \phi((b, d))}{b \phi(b)} 
\ll& 
\log N 
\sum_{d \leq N^\theta} 
\frac{\mu^2(d)}{\phi(d)}
\sum_{b > B}
\frac{\mu^2(b) (b, d)}{b \phi(b)}. 
\end{align*}
\normalsize
We apply $\displaystyle (b, d) = \sum_{r \mid(b, d)} \phi(r)$, and then make the  substitution $b = b_1 r$ and $d = d_1 r$ to obtain that
\begin{align*}
&\sum_{d \leq N^\theta} 
\frac{\mu^2(d)}{\phi(d)}
\sum_{b > B}
\frac{\mu^2(b)(b, d)}{b \phi(b)} 
=
\sum_{d \leq N^\theta} 
\frac{\mu^2(d)}{\phi(d)}
\sum_{b > B}
\frac{\mu^2(b)}{b \phi(b)} 
\sum_{r \mid (b, d)} 
\phi(r) \\
&\ll 
\sum_{r \leq N^\theta} 
\frac{1}{r \phi(r)}
\sum_{d_1 \leq \frac{N^{\theta}}{r}}
\frac{1}{\phi(d_1)}
\sum_{b_1 > \frac{B}{r}}
\frac{1}{b_1 \phi(b_1)} 
\ll \frac{(\log N)^2 \log \log N}{B},
\end{align*}
where in the last line some summation results of~\cite{Sita} are applied. 
Since $B = (\log N)^{A + 4}$, we have
\begin{align*}
&\sum_{d \leq N^\theta} \mu(d)  \log(1/d)
\sum_{\substack{b \leq B \\ (b d, N) = 1}}
\frac{\mu(b)}{\phi([b^2, d])} \\
&= 
\sum_{d \leq N^\theta} 
\frac{\mu(d)  \log(1/d)}{\phi(d)}
\sum_{\substack{b \geq 1 \\ (b d, N) = 1}}
\frac{\mu(b) \phi((b, d))}{b \phi(b)} 
+ O \left( \frac{1}{(\log N)^A} \right) \\ 
&=
\sum_{\substack{d \leq N^\theta \\ (d, N) = 1}} 
\frac{\mu(d)  \log(1/d)}{\phi(d)}
\prod_{p \nmid N}
\left( 
1 - \frac{\phi((p, d))}{p(p - 1)}
\right)
+ O \left( \frac{1}{(\log N)^A} \right) \\
&= 
\mathcal{A}(N) 
\sum_{\substack{d \leq N^\theta \\ (d, N) = 1}} 
\frac{\mu(d)  \log(1/d) g_N(d)}{\phi(d)}
+ O \left( \frac{1}{(\log N)^A} \right), 
\end{align*}
where $\mathcal{A}(N)$ is defined in~\eqref{Ah} and $g_N(d)$ is defined in~\eqref{gNd}. 
Therefore, it follows by Lemma~\ref{serieslemma} that
\begin{align*}
&\sum_{d \leq N^\theta} \mu(d)  \log(1/d)
\sum_{\substack{b \leq B \\ (b d, N) = 1}}
\frac{\mu(b)}{\phi([b^2, d])}\\
&= \mathfrak{S}(N)
+ O\left(e^{-c_2 \sqrt{\log (N^\theta)}} \right)
+ O \left( \frac{1}{(\log N)^A} \right)\\
&= \mathfrak{S}(N)
+ O \left( \frac{1}{(\log N)^A} \right).
\end{align*}
\end{proof}

\begin{lemma}\label{E4ylemma}
Let $A > 0$ and $0 < \theta < 1$ be two positive real numbers. Suppose the conjecture $EH_\mu(N^{1 - \theta})$ holds, then we have 
\begin{align*} 
\sum_{q \leq N^{1 - \theta}} 
\max_{y < N}
\max_{(a, q) = 1}
& \bigg|
\sum_{\substack{n \leq y \\ n \equiv a \pmod{q}}}
\Lambda(n) 
\mu(N - n)
\log (N - n) \\
& -
\frac{1}{\phi(q)}
\sum_{n \leq y}
\Lambda(n) 
\mu(N - n)
\log (N - n)
\bigg| 
\ll \frac{N}{(\log N)^A}.  
\end{align*}
\end{lemma}
\begin{proof}
The proof is similar to the proof of~\cite[Proposition 4.2]{MuV}. 
For any $(a, q) = 1$ and $y < N$, denote 
\[
\sum_{\substack{n \leq y \\ n \equiv a \pmod{q}}}
\Lambda(n) \mu(N - n)
= M(y) + E_\mu(y, q, a), 
\]
where 
\[
M(y) = \frac{1}{\phi(q)}
\sum_{n \leq y} \Lambda(n) \mu(N - n). 
\]
By partial summation, we have
\small
\begin{align*}
&\sum_{\substack{n \leq y \\ n \equiv a \pmod{q}}}
\Lambda(n) \mu(N - n) \log(N - n) \\
&= 
\left(
\sum_{\substack{n \leq y \\ n \equiv a \pmod{q}}}
\Lambda(n) \mu(N - n)
\right)
\log (N - y) 
+ \int_2^y \frac{1}{N - t} 
\left(
\sum_{\substack{n \leq t \\ n \equiv a \pmod{q}}}
\Lambda(n) \mu(N - n) 
\right) \, dt \\
&= M(y)\log(N - y)
+ \int_2^y \frac{M(t)}{N - t} \, dt 
+ O \left(
\max_{t \leq y}
|E_\mu(t, q, a)|
\log N
\right). 
\end{align*}
\normalsize
Using partial summation again, the sum of the first two terms above is
\[
\frac{1}{\phi(q)}
\sum_{n \leq y}
\Lambda(n) \mu(N - n) \log(N - n). 
\]
Now the lemma immediately follows from our assumption  $EH_\mu(N^{1 - \theta})$. 
\end{proof}

\section{Proof of Theorem~\ref{maintheorem} and Corollary~\ref{coro}}
\label{proofofmaintheorem}
By the M\"{o}bius inversion formula, we have 
\begin{equation}\label{Lambda}
\Lambda(n)
= \sum_{d \mid n} \mu(d) \log (1/d). 
\end{equation}
As in~\eqref{R0}-\eqref{R3}, we split the sum on the right hand side of~\eqref{Lambda} into $\Lambda_{\alpha}(n)$ and $\widetilde{\Lambda}_{\alpha}(n)$, where  
\[
\Lambda_{\alpha}(n) = \sum_{\substack{d \mid n \\ d \leq {\alpha}}} \mu(d) \log(1/d), 
\]
\[
\widetilde{\Lambda}_{\alpha}(n) = \sum_{\substack{d \mid n \\ d > {\alpha}}} \mu(d) \log(1/d), 
\]
and ${\alpha}$ is a real number that will be determined later. Therefore, we have 
\begin{equation}\label{nNminusn}
\sum_{n < N}
\Lambda(n) \Lambda(N - n)
= \sum_{n < N} \Lambda(n) \Lambda_{\alpha}(N - n) 
+ \sum_{n < N} \Lambda(n) \widetilde{\Lambda}_{\alpha}(N - n). 
\end{equation}
Note that if we take ${\alpha} = Q = N^{\frac{1}{2}}(\log N)^{-20}$, then Pan's result implies the first sum on the right hand side of~\eqref{nNminusn} is asymptotic to $R_0(N) + R_2(N) + O \left( \frac{N}{\log N} \right) = \mathfrak{S}(N) N  + O \left( \frac{N}{\log N} \right)$ and the second sum on the right hand side of~\eqref{nNminusn} is asymptotic to $R_1(N) + R_3(N) = R_3(N) + O \left( \frac{N}{\log N} \right)$. This indicates that the sum $\sum_{n < N} \Lambda(n) \widetilde{\Lambda}_{\alpha}(N - n)$ is harder to deal with, which represents the main obstacle we are facing. 

To make the estimate of $\sum_{n < N} \Lambda(n) \widetilde{\Lambda}_{\alpha}(N - n)$ more accessible, we only sum up those $n < N$ such that $N - n$ is square-free, as has been done  on~\cite[Page 647]{MuV}. Since $\Lambda(N - n) = \mu^2(N - n) \Lambda(N - n)$ except when $n = N - p^{\ell_1} > 0$ for some prime~$p$ and integer $\ell_1 \geq 2$, and there are only $O\left(N^{\frac{1}{2}} \log N\right)$ such $n < N$, we can write~\eqref{nNminusn} as
\small
\begin{align}\label{S1S2}
&\sum_{n < N}
\Lambda(n) \Lambda(N - n) \nonumber \\
&= \sum_{n < N} \Lambda(n) \mu^2(N - n) \Lambda(N - n) + O\left(N^{\frac{1}{2}} (\log N)^3\right) \nonumber \\
&= \sum_{n < N} \Lambda(n) \mu^2(N - n) \Lambda_{\alpha}(N - n)
+ \sum_{n < N} \Lambda(n) \mu^2(N - n) \widetilde{\Lambda}_{\alpha}(N - n)
+ O\left(N^{\frac{1}{2}} (\log N)^3\right) \nonumber \\
&= S_1({\alpha}) + S_2({\alpha}) + O\left(N^{\frac{1}{2}} (\log N)^3\right), 
\end{align}
\normalsize
where
\[
S_1({\alpha}) 
= 
\sum_{n < N} \Lambda(n) \mu^2(N - n) \sum_{\substack{d \mid (N - n) \\ d \leq {\alpha}}} \mu(d) \log(1/d)
\]
and 
\begin{align*}
S_2({\alpha}) 
=& 
\sum_{n < N} \Lambda(n) \mu^2(N - n)
\sum_{\substack{d \mid (N - n) \\ d > {\alpha}}} \mu(d) \log(1/d) \\
=&
\sum_{n < N} \Lambda(n) \mu^2(N - n)
\sum_{\substack{k \mid (N - n) \\ k < \frac{N - n}{{\alpha}}}} \mu\left(\frac{N - n}{k}\right) \log\left(\frac{k}{N - n}\right)
.
\end{align*}

The $\mu^2(N - n)$ term, i.e., the assumption that we only deal with square-free $N - n$, will make it easier when we evaluate $S_2({\alpha})$, since we can apply the equation $\mu^2(N - n) \mu\left( \frac{N - n}{k} \right) = \mu(N - n) \mu(k)$ later in~\eqref{e4.1} when $N - n$ is square-free and $k \mid (N - n)$. Meanwhile, the term $\mu^2(N - n)$ does not change the asymptotic behavior of $S_1({\alpha})$ as shown in the next subsection. 

\subsection{Evaluation of $S_1({\alpha})$ using $EH(N^{\theta} (\log N)^{2A + 8})$} \label{S1section}
In this subsection we prove that \begin{equation}\label{S1yresult}
S_1({\alpha}) = \mathfrak{S}(N) N + O\left( \frac{N}{(\log N)^A} \right)
\end{equation}
under the assumption $EH(N^{\theta} (\log N)^{2A+8})$. 

Since
\[
\mu^2(N - n) = \sum_{b^2 \mid (N - n)} \mu(b)
= \begin{cases}
1, & \text{ if $N - n$ is square-free,}\\
0, &\text{ otherwise,}
\end{cases}
\]
we have
\begin{align}\label{S1}
S_1({\alpha})
=& \sum_{n < N} \Lambda(n) \mu^2(N - n) 
\sum_{\substack{d \mid (N - n) \\ d \leq {\alpha}}} \mu(d) \log(1/d) \nonumber \\
=& 
\sum_{n < N} \Lambda(n) 
\sum_{b^2 \mid (N - n)} \mu(b)
\sum_{\substack{d \mid (N - n) \\ d \leq {\alpha}}}
\mu(d) \log(1/d) \nonumber \\
=& 
\sum_{d \leq {\alpha}} \mu(d) \log(1/d)
\sum_{\substack{b < \sqrt{N} \\ [b^2, d] < N}}
\mu(b)
\sum_{\substack{n < N \\ n \equiv N \pmod{[d, b^2]}}} \Lambda(n). 
\end{align}
We truncate the sum over $b < \sqrt{N}$ into two parts $b \leq B$ and $B < b < \sqrt{N}$ with $B = (\log N)^{A + 4}$. Let $(b, d) = \delta_1$, $b = b_2 \delta_1$, and $d = d_2 \delta_1$. Then $[b^2, d] = b_2^2 d_2 \delta_1^2$. For the sum over $B < b < \sqrt{N}$, we have 
\begin{align*}
& \sum_{d \leq {\alpha}} \mu(d) \log(1/d)
\sum_{\substack{B < b < \sqrt{N} \\ [b^2, d] < N }}
\mu(b)
\sum_{\substack{n < N \\ n \equiv N \pmod{[b^2, d]}}} \Lambda(n) \\
&\ll 
(\log N)^2
\sum_{d < N} 
\sum_{\substack{b > B \\ [b^2, d] < N}}
\left( \frac{N}{[b^2, d]} + 1 \right) \\
&\ll 
(\log N)^2
\sum_{\delta_1 < N}
\sum_{d_2 < \frac{N}{\delta_1}} 
\sum_{\substack{b_2 > \frac{B}{\delta_1} \\ (b_2)^2 d_2 (\delta_1)^2 \leq N}}
\frac{N}{(b_2)^2 d_2 (\delta_1)^2} \\
&\ll 
N (\log N)^2
\sum_{d_2 \leq N} \frac{1}{d_2} 
\sum_{b_2 \leq \sqrt{N}} \frac{1}{(b_2)^2}
\sum_{\delta_1 > \frac{B}{b_2}}
\frac{1}{(\delta_1)^2} \\
&\ll \frac{N (\log N)^4}{B} 
\ll \frac{N}{(\log N)^A}. 
\end{align*}
Thus we can write~\eqref{S1} as
\begin{equation}\label{bdsum}
S_1({\alpha}) = 
\sum_{d \leq {\alpha}} \mu(d) \log(1/d)
\sum_{\substack{b \leq B \\ [b^2, d] < N}} \mu(b)
\sum_{\substack{n < N \\ n \equiv N \pmod{[b^2, d]}}} \Lambda(n)
+ O\left( \frac{N}{(\log N)^A} \right). 
\end{equation}
For $b$ or $d$ that are not relatively prime to $N$ on the right hand side of~\eqref{bdsum}, suppose $(b d, N) = \delta_2 > 1$. If an integer $n < N$ satisfies $\Lambda(n) \neq 0$ and $n \equiv N \pmod{[b^2, d]}$, then $n$ is a prime power $n = p^{\ell_2}$ for some integer $\ell_2 \geq 1$, and moreover, we have $n \equiv p^{\ell_2} \equiv \delta_2 \frac{N}{\delta_2} \pmod{\delta_2 \frac{[b^2, d]}{\delta_2}}$. This implies $p \mid \delta_2$ and thus $p \mid N$. Therefore, taking ${\alpha} = N^{\theta}$ for a fixed constant $0 < \theta < 1$, we know 
\begin{align*}
& \sum_{d \leq {\alpha}} \mu(d) \log(1/d)
\sum_{\substack{b \leq B \\ [b^2, d] < N \\ (b d, N) > 1}} \mu(b)
\sum_{\substack{n < N\\ n \equiv N \pmod{[b^2, d]}}} \Lambda(n) \\
&\ll
\log N
\sum_{d \leq {\alpha}}  
\sum_{\substack{b \leq B \\ [b^2, d] < N \\ (b d, N) > 1}} 
\sum_{p \mid N}
\sum_{\ell_2 \leq \frac{\log N}{\log p}}
\Lambda\left(p^{\ell_2}\right) \\
&\ll {\alpha} B (\log N)^2
\ll \frac{N}{(\log N)^A}. 
\end{align*}
Note that we can drop the restriction $[b^2, d] < N$ when $b \leq B = (\log N)^{A + 4}$ and $d \leq {\alpha} = N^{\theta}$. Therefore, we have
\begin{align}\label{withE}
S_1({\alpha}) = & 
\sum_{d \leq {\alpha}} \mu(d) \log(1/d)
\sum_{\substack{b \leq B \\ (b d, N) = 1}} \mu(b)
\sum_{\substack{n < N \\ n \equiv N \pmod{[b^2, d]}}} \Lambda(n)
+ O\left( \frac{N}{(\log N)^A} \right) \nonumber \\
= & 
N \sum_{d \leq {\alpha}} \mu(d) \log(1/d)
\sum_{\substack{b \leq B \\ (b d, N) = 1}} \frac{\mu(b)}{\phi([b^2, d])} 
+ E_1({\alpha}) + O\left( \frac{N}{(\log N)^A} \right), 
\end{align}
where
\begin{align*}
E_1({\alpha}) =& 
\sum_{d \leq {\alpha}} \mu(d) \log(1/d)
\sum_{\substack{b \leq B \\ (b d, N) = 1}} \mu(b)
\left(
\sum_{\substack{n < N \\ n \equiv N \pmod{[b^2, d]}}} \Lambda(n)
- \frac{N}{\phi([b^2, d])}
\right) \\
\ll&
\log N
\sum_{d \leq {\alpha}} \mu^2(d)
\sum_{\substack{b \leq B \\ (b d, N) = 1}} \mu^2(b)
\left|
\sum_{\substack{n < N \\ n \equiv N \pmod{[b^2, d]}}} \Lambda(n)
- \frac{N}{\phi([b^2, d])}
\right|. 
\end{align*}
For $b \leq B$ and $d \leq {\alpha}$ that are square-free, let $z = [b^2, d]$. Since the number of pairs of square-free natural numbers $b$ and $d$ such that $[b^2, d] = z$ is bounded by $\tau_3(z)$, the number of ways of writing $z$ as a product of three positive integers,
we can bound the error term $E_1(\alpha)$ as
\begin{align*}
E_1({\alpha}) 
\ll
\log N
\sum_{z \leq {\alpha} B^2} \tau_3(z)
\left|
\sum_{\substack{n < N\\ n \equiv N \pmod{z}}} \Lambda(n)
- \frac{N}{\phi(z)}
\right|. 
\end{align*}
For all $z \leq {\alpha} B^2 = N^\theta (\log N)^{2A + 8}$, since $\displaystyle \sum_{\substack{n < N\\ n \equiv N \pmod{z}}} \Lambda(n) \ll \frac{N \log N}{z}$ and $\frac{N}{\phi(z)} \ll \frac{N \log \log z}{z} \ll \frac{N \log N}{z}$, we have a trivial bound 
\begin{equation*}
\left| 
\sum_{\substack{n < N\\ n \equiv N \pmod{z}}}
\Lambda(n) - \frac{N}{\phi(z)}
\right|
\ll \frac{N \log N}{z}.
\end{equation*}
By the Cauchy-Schwarz inequality, we know
\begin{align*}
E_1({\alpha}) 
\ll
N^{\frac{1}{2}} (\log N)^{\frac{3}{2}}
\left( 
\sum_{z \leq {\alpha} B^2} 
\frac{\tau_3^2(z)}{z}
\right)^{\frac{1}{2}}
\left( 
\sum_{z \leq {\alpha} B^2}
\left|
\sum_{\substack{n < N \\ n \equiv N \pmod{z}}}
\Lambda(n)
- \frac{N}{\phi(z)}
\right|
\right)^{\frac{1}{2}}. 
\end{align*}
By~\cite[(1.80)]{IK} and partial summation, we can bound the first sum by 
\[
\sum_{z \leq {\alpha} B^2} 
\frac{\tau_3^2(z)}{z}
\ll (\log N)^{9}. 
\]
By our assumption $EH(N^\theta (\log N)^{2A+8})$, we bound the second sum by
\small
\begin{align*}
\sum_{z \leq {\alpha} B^2}
\left|
\sum_{\substack{n < N \\ n \equiv N \pmod{z}}}
\Lambda(n)
- \frac{N}{\phi(z)}
\right| 
&\ll
\sum_{q \leq N^\theta (\log N)^{2 A + 8}}
\max_{y \leq N}
\max_{\substack{a \\ (a, q) = 1}}
\left|
\sum_{\substack{n \leq y \\ n \equiv a \pmod{q}}}
\Lambda(n)
- \frac{y}{\phi(q)}
\right| \\
&\ll
\frac{N}{(\log N)^{2A+12}}. 
\end{align*}
\normalsize
Therefore, we obtain 
\[
E_1({\alpha}) \ll 
N^{\frac{1}{2}} (\log N)^{\frac{3}{2}} 
\left( (\log N)^{9}\right)^{\frac{1}{2}} 
\left( \frac{N}{(\log N)^{2A+12}} \right)^{\frac{1}{2}}
\ll \frac{N}{(\log N)^A}. 
\]
Finally, inserting the above bound into~\eqref{withE}, we have
\[
S_1 ({\alpha})
= N \sum_{d \leq {\alpha}} \mu(d) \log(1/d)
\sum_{\substack{b \leq B \\ (b d, N) = 1}} \frac{\mu(b)}{\phi([b^2, d])} 
+ O\left( \frac{N}{(\log N)^A} \right),
\]
which implies~\eqref{S1yresult} after applying Lemma~\ref{S11}.

\subsection{Evaluation of $S_2({\alpha})$ using $EH_\mu(N^{1 - \theta})$} \label{S2section}
In this subsection we show that
\begin{equation}\label{S2yresult}
S_2({\alpha}) = 
- \mathfrak{S}(N) \sum_{n < N}
\Lambda(n) \mu(N - n)
+ 
O\left( \frac{N}{(\log N)^A} \right)
\end{equation} under the assumption  $EH_\mu(N^{1 - \theta})$. 

Suppose $N - n$ is square-free and let $k$ be a divisor of $N - n$ such that $k < \frac{N - n}{{\alpha}}$. Since $\mu^2(N - n) \mu\left( \frac{N - n}{k} \right) = \mu(N - n) \mu(k)$, we have 
\begin{align} \label{e4.1}
S_2({\alpha}) 
=&
\sum_{n < N} \Lambda(n) \mu^2(N - n) \sum_{\substack{d \mid (N - n) \\ d > {\alpha}}} \mu(d) \log(1/d) \nonumber \\
=&
\sum_{n < N} \Lambda(n) \mu^2(N - n) \sum_{\substack{k \mid (N - n) \\ k < \frac{N - n}{{\alpha}}}} \mu\left(\frac{N -n}{k}\right) \log\left(\frac{k}{N - n}\right) \nonumber \\
=& 
\sum_{n < N} \Lambda(n) \sum_{\substack{k \mid (N - n) \\ k < \frac{N - n}{{\alpha}}}} \mu(N - n) \mu(k) \log\left( \frac{k}{N - n} \right) \nonumber \\
=& 
\sum_{k < \frac{N - 1}{{\alpha}}} \mu(k)
\sum_{\substack{n < N\\ n \equiv N \pmod{k}}}
\Lambda(n) 
\mu(N - n)
\log\left( \frac{k}{N - n} \right). 
\end{align}
For those integers $k < \frac{N - 1}{{\alpha}}$ such that $(k, N) = \delta_3 > 1$, suppose an integer $n < N$ satisfies $\Lambda(n) \neq 0$ and $n \equiv N \pmod{k}$, then $n$ is a prime power $n = p^{\ell_3}$ for some integer $\ell_3 \geq 1$, and the condition $n \equiv p^{\ell_3} \equiv N \pmod{k}$ would imply $p \mid \delta_3$ and hence $p \mid N$. Therefore, we have 
\begin{align*}
&\sum_{\substack{k < \frac{N - 1}{{\alpha}} \\ (k, N) > 1}} \mu(k)
\sum_{\substack{n < N \\ n \equiv N \pmod{k}}}
\Lambda(n) 
\mu(N - n)
\log\left( \frac{k}{N - n} \right)\\
&\ll
\log N 
\sum_{k < \frac{N}{{\alpha}}} 
\sum_{p \mid N}
\sum_{\ell_3 \leq \frac{\log N}{\log p}}
\Lambda\left(p^{\ell_3}\right) 
\ll
\frac{N}{{\alpha}} (\log N)^2, 
\end{align*}
which implies
\begin{align}\label{S3-S4}
S_2({\alpha}) =&
\sum_{\substack{k < \frac{N - 1}{{\alpha}} \\ (k, N) = 1}} \mu(k)
\sum_{\substack{n < N\\ n \equiv N \pmod{k}}}
\Lambda(n) 
\mu(N - n)
\log\left( \frac{k}{N - n} \right)
+ O\left( \frac{N}{(\log N)^A} \right) \nonumber \\
=& 
S_3({\alpha})
- S_4({\alpha}) + O\left( \frac{N}{(\log N)^A} \right), 
\end{align}
where
\[
S_3({\alpha}) = 
\sum_{\substack{k < \frac{N - 1}{{\alpha}} \\ (k, N) = 1}} 
\mu(k) \log k
\sum_{\substack{n < N\\ n \equiv N \pmod{k}}}
\Lambda(n) 
\mu(N - n)
\]
and 
\[
S_4({\alpha}) = 
\sum_{\substack{k < \frac{N - 1}{{\alpha}} \\ (k, N) = 1}} \mu(k)
\sum_{\substack{n < N \\ n \equiv N \pmod{k}}}
\Lambda(n) 
\mu(N - n)
\log (N - n). 
\]
To estimate $S_3({\alpha})$, we first write 
\begin{align*}
S_3({\alpha})=
\sum_{\substack{k < \frac{N - 1}{{\alpha}} \\ (k, N) = 1}} 
\frac{\mu(k) \log k}{\phi(k)}
\sum_{n < N}
\Lambda(n) \mu(N - n) 
+ E_3({\alpha}), 
\end{align*}
where
\small
\begin{align*}
E_3({\alpha}) = \sum_{\substack{k < \frac{N - 1}{{\alpha}} \\ (k, N) = 1}} 
\mu(k) \log k
\left(
\sum_{\substack{n < N\\ n \equiv N \pmod{k}}}
\Lambda(n) 
\mu(N - n)
-
\frac{1}{\phi(k)}
\sum_{n < N}
\Lambda(n) \mu(N - n)
\right). 
\end{align*}
\normalsize
Our assumption $EH_\mu(N^{1 - \theta})$ implies that 
\small
\begin{align*}
E_3({\alpha}) 
\ll&
\log N
\sum_{\substack{k < \frac{N - 1}{{\alpha}} \\ (k, N) = 1}} 
\left|
\sum_{\substack{n < N\\ n \equiv N \pmod{k}}}
\Lambda(n) 
\mu(N - n)
-
\frac{1}{\phi(k)}
\sum_{n < N}
\Lambda(n) \mu(N - n)
\right| \\
\ll& 
\log N
\sum_{q \leq N^{1 - \theta}}
\max_{y < N}
\max_{\substack{a \\ (a, q) = 1}}
\left|
\sum_{\substack{n \leq y \\ n \equiv a \pmod{q}}}
\Lambda(n) 
\mu(N - n)
-
\frac{1}{\phi(q)}
\sum_{n \leq y}
\Lambda(n) \mu(N - n)
\right| \\
\ll & \frac{N}{(\log N)^A}. 
\end{align*}
\normalsize
Therefore, we have 
\[
S_3({\alpha})=
\sum_{\substack{k < \frac{N - 1}{{\alpha}} \\ (k, N) = 1}} 
\frac{\mu(k) \log k}{\phi(k)}
\sum_{n < N}
\Lambda(n) \mu(N - n) 
+ O\left( \frac{N}{(\log N)^A} \right). 
\]
By Lemma~\ref{GYlemma}, we may write the sum on the right hand side as
\begin{align*}
&\left( 
- \mathfrak{S}(N)
+ O\left( 
e^{
-c_1 \sqrt{\log \left(N / \alpha \right) }
}\right)
\right)
\sum_{n < N}
\Lambda(n) \mu(N - n)\\
&=
- \mathfrak{S}(N)
\sum_{n < N}
\Lambda(n) \mu(N - n)
+ O\left( 
\frac{N}{
e^{
c_1 \sqrt{\log \left(N / \alpha \right) }
}
}
\right) \\
&=
- \mathfrak{S}(N)
\sum_{n < N}
\Lambda(n) \mu(N - n)
+ O\left( \frac{N}{(\log N)^{A}} \right),
\end{align*}
where in the second line the trivial bound $\sum_{n < N} |\Lambda(n) \mu(N - n)| = O(N)$ is applied in the error term and $c_1$ is a positive constant. So we arrive at the estimate
\begin{equation}\label{S2ywithE2y}
    S_3(\alpha)=- \mathfrak{S}(N)
\sum_{n < N}
\Lambda(n) \mu(N - n)
+ O\left( \frac{N}{(\log N)^{A}} \right).
\end{equation}

Next we treat $S_4({\alpha})$ in a similar fashion. Again, we write 
\begin{align*}
S_4({\alpha}) 
=&
\sum_{\substack{k < \frac{N - 1}{{\alpha}} \\ (k, N) = 1}} \mu(k)
\sum_{\substack{n < N \\ n \equiv N \pmod{k}}}
\Lambda(n) 
\mu(N - n)
\log (N - n) \nonumber \\
=&
\sum_{\substack{k < \frac{N - 1}{{\alpha}} \\ (k, N) = 1}} \frac{\mu(k)}{\phi(k)}
\sum_{n < N}
\Lambda(n) 
\mu(N - n)
\log (N - n) + E_4({\alpha}), 
\end{align*}
where 
\scriptsize
\begin{align*} 
E_4({\alpha}) =& 
\sum_{\substack{k < \frac{N - 1}{{\alpha}} \\ (k, N) = 1}} \mu(k)
\left(
\sum_{\substack{n < N \\ n \equiv N \pmod{k}}}
\Lambda(n) 
\mu(N - n)
\log (N - n)
-
\frac{1}{\phi(k)}
\sum_{n < N}
\Lambda(n) 
\mu(N - n)
\log (N - n)
\right). 
\end{align*}
\normalsize
By Lemma~\ref{E4ylemma}, we have 
\scriptsize
\begin{align*} 
& E_4({\alpha})
\ll
\sum_{\substack{k < \frac{N - 1}{{\alpha}} \\ (k, N) = 1}} 
\left|
\sum_{\substack{n < N \\ n \equiv N \pmod{k}}}
\Lambda(n) 
\mu(N - n)
\log (N - n)
-
\frac{1}{\phi(k)}
\sum_{n < N}
\Lambda(n) 
\mu(N - n)
\log (N - n)
\right| \nonumber \\
&\ll
\sum_{q \leq N^{1 - \theta}} 
\max_{y < N}
\max_{(a, q) = 1}
\left|
\sum_{\substack{n \leq y \\ n \equiv a \pmod{q}}}
\Lambda(n) 
\mu(N - n)
\log (N - n)
-
\frac{1}{\phi(q)}
\sum_{n \leq y}
\Lambda(n) 
\mu(N - n)
\log (N - n)
\right| \nonumber \\
&\ll \frac{N}{(\log N)^A}. 
\end{align*}
\normalsize
Moreover, it follows by Lemma~\ref{GYlemma} that
\[
\left(
\sum_{\substack{k < \frac{N - 1}{{\alpha}} \\ (k, N) = 1}} \frac{\mu(k)}{\phi(k)}
\right)
\left(
\sum_{n < N}
\Lambda(n) 
\mu(N - n)
\log (N - n)
\right) \ll 
\frac{N \log N}
{e^{c_1 \sqrt{\log \left(N / {\alpha} \right)}}} \ll
 \frac{N}{(\log N)^A },
\]
where $c_1$ is a positive constant.

Now, merging the above two estimates yields
\begin{align}\label{E2ybound}
S_4({\alpha}) 
= O \left( \frac{N}{(\log N)^A } \right). 
\end{align}
Therefore, we deduce the desired estimate~\eqref{S2yresult} from~\eqref{S3-S4}-\eqref{E2ybound}.

\subsection{Proof of Theorem~\ref{maintheorem}}
Recall from~\eqref{S1S2} that
\begin{align*}
\sum_{n < N} \Lambda(n) \Lambda(N - n) 
= S_1({\alpha}) + S_2({\alpha}) + O\left(N^{\frac{1}{2}} (\log N)^3\right).
\end{align*}
Moreover, for any $A>0$, if we suppose that $EH(N^\theta(\log N)^{2A+8})$ and $EH_\mu(N^{1 - \theta})$ hold for some constant $0 < \theta < 1$ , then we recall from~\eqref{S1yresult} and~\eqref{S2yresult} that
\[
S_1({\alpha}) = \mathfrak{S}(N) N + O\left( \frac{N}{(\log N)^A} \right)
\]
and
\[
S_2({\alpha}) = 
- \mathfrak{S}(N) \sum_{n < N}
\Lambda(n) \mu(N - n)
+ 
O\left( \frac{N}{(\log N)^A} \right). 
\]
Therefore, we have
\begin{align}\label{LambdanLambdaN-n}
\sum_{n < N} \Lambda(n) \Lambda(N - n) 
= \mathfrak{S}(N)
\left(
N - \sum_{n < N} \Lambda(n) \mu(N - n)
\right)
+ O\left( \frac{N}{(\log N)^A} \right). 
\end{align}
Thus to find a lower bound on $\sum_{n < N} \Lambda(n) \Lambda(N - n)$, we only need an upper bound on $\displaystyle \left|\sum_{n < N} \Lambda(n) \mu(N - n)\right|$. We first apply the triangle inequality and obtain that
\[
\left| 
\sum_{n < N}
\Lambda(n) \mu(N - n)
\right|
\leq 
\sum_{n < N}
\Lambda(n) \mu^2(N - n). 
\]
Since $\displaystyle \mu^2(N - n) = \sum_{b^2 \mid (N - n)} \mu(b)$, we have 
\begin{align}\label{lastbound}
& \sum_{n < N}
\Lambda(n) \mu^2(N - n) \nonumber \\
&= 
\sum_{n < N}
\Lambda(n)
\sum_{b^2 \mid (N - n)}
\mu(b) \nonumber \\
&=
\sum_{b < (N - 1)^{\frac{1}{2}}} \mu(b)
\sum_{\substack{n < N \\ n \equiv N \pmod{b^2}}}
\Lambda(n) \nonumber \\
&= 
\sum_{b \leq N^{\frac{1}{4} - \epsilon}} \mu(b)
\sum_{\substack{n < N \\ n \equiv N \pmod{b^2}}}
\Lambda(n)
+
\sum_{N^{\frac{1}{4} - \epsilon} < b < N^{\frac{1}{2}}} \mu(b)
\sum_{\substack{n < N \\ n \equiv N \pmod{b^2}}}
\Lambda(n).
\end{align}
The sum over $N^{\frac{1}{4} - \epsilon} < b < N^{\frac{1}{2}}$ can be bounded by 
\begin{equation}\label{bigpart}
\sum_{N^{\frac{1}{4} - \epsilon} < b < N^{\frac{1}{2}}} \mu(b)
\sum_{\substack{n < N \\ n \equiv N \pmod{b^2}}}
\Lambda(n)
\ll 
\log N \sum_{N^{\frac{1}{4} - \epsilon} < b \leq N^{\frac{1}{2}}} \frac{N}{b^2}
\ll N^{\frac{3}{4} + \epsilon}. 
\end{equation}
The contribution of those $b\le N^{\frac14-\epsilon}$ such that $(b, N) > 1$ is  
\begin{equation}\label{smallpart}
   \sum_{\substack{b \leq N^{\frac{1}{4} - \epsilon} \\ (b, N) >1}} \mu(b)
\sum_{\substack{n < N\\ n \equiv N \pmod{b^2}}}
\Lambda(n) \ll N^{\frac{1}{4} - \epsilon}.
\end{equation}
On the other hand, for those $b\leq N^{\frac{1}{4} - \epsilon}$ with $(b, N) = 1$, we apply the Bombieri-Vinogradov theorem and obtain
\begin{align}\label{BVbound}
&\sum_{\substack{b \leq N^{\frac{1}{4} - \epsilon} \\ (b, N) = 1}} \mu(b)
\sum_{\substack{n < N\\ n \equiv N \pmod{b^2}}}
\Lambda(n) \nonumber \\
&=
N \sum_{\substack{b \leq N^{\frac{1}{4} - \epsilon} \\ (b, N) = 1}}
\frac{\mu(b)}{\phi(b^2)}
+  
\sum_{\substack{b \leq N^{\frac{1}{4} - \epsilon} \\ (b, N) = 1}}
\mu(b)
\left( 
\sum_{\substack{n < N\\ n \equiv N \pmod{b^2}}}
\Lambda(n)
 - \frac{N}{\phi(b^2)}
\right) \nonumber \\
&=
N \sum_{\substack{b \leq N^{\frac{1}{4} - \epsilon} \\ (b, N) = 1}}
\frac{\mu(b)}{\phi(b^2)}
+  
O \left(
\sum_{\substack{b \leq N^{\frac{1}{4} - \epsilon} \\ (b, N) = 1}}
\left|
\sum_{\substack{n < N\\ n \equiv N \pmod{b^2}}}
\Lambda(n)
 - \frac{N}{\phi(b^2)}
\right|
\right) \nonumber \\
&= 
N \sum_{\substack{b \leq N^{\frac{1}{4} - \epsilon} \\ (b, N) = 1}}
\frac{\mu(b)}{\phi(b^2)}
+  
O \left(
\frac{N}{(\log N)^{A+1}}
\right). 
\end{align}
It is easy to see that the infinite series 
\begin{equation}\label{ANseries}
\sum_{(d, N) = 1} \frac{\mu(d)}{\phi(d^2)}
= \prod_{\substack{p \nmid N \\ p > 2}} \left( 1 - \frac{1}{p(p - 1)} \right)
= \mathcal{A}(N)
\end{equation}
converges
and that  its tail is bounded by 
\begin{equation}\label{tailbound}
\sum_{\substack{b > N^{\frac{1}{4} - \epsilon} \\ (d, N) = 1}} \frac{\mu(d)}{\phi(d^2)}
\ll \sum_{b > N^{\frac{1}{4} - \epsilon} } \frac{\log \log d}{d^2}
\ll N^{- \frac{1}{4} + \epsilon}. 
\end{equation}
Therefore, by~\eqref{lastbound}-\eqref{tailbound}, we have
\begin{align*} 
\sum_{n < N}
\Lambda(n) \mu^2(N - n)
= N \mathcal{A}(N) 
+ 
O \left(
\frac{N}{(\log N)^{A+1}}
\right). 
\end{align*}
Thus, we deduce from~\eqref{LambdanLambdaN-n} that
\begin{align*}
\sum_{n < N} \Lambda(n) \Lambda(N - n)
\geq&
\mathfrak{S}(N) \left(N - N \mathcal{A}(N) 
+ O \left( \frac{N}{(\log N)^{A+1}} \right)
\right) 
+ O \left( \frac{N}{(\log N)^A} \right)\\
\geq&\mathfrak{S}(N) 
\left(1 - \mathcal{A}(N) \right) N
+ O \left( \frac{N}{(\log N)^A} \right),
\end{align*}
where in the last line we have applied the bound
\begin{align*}
\mathfrak{S}(N) 
=& 2 \prod_{p > 2} \left( 1 - \frac{1}{(p - 1)^2} \right)
\prod_{\substack{p \mid N \\ p > 2}} \left( 1 + \frac{1}{p - 2} \right) \\
\ll&
\prod_{2 < m \leq \log N}
\frac{m - 1}{m - 2} \\
\ll& \log N.
\end{align*}

Finally, it follows directly from~\eqref{LambdanLambdaN-n} that the assertions
\[
\sum_{n < N} \Lambda(n) \Lambda(N - n) 
\sim 
\mathfrak{S}(N) N
\]
and
\[
\sum_{n < N} \Lambda(n) \mu(N - n) = o(N)
\]
are equivalent. This completes the proof of Theorem~\ref{maintheorem}. 

\subsection{Proof of Corollary~\ref{coro}}

To prove the corollary, we just take $A=2+\frac{\epsilon}2$ in our theorem and obtain 
\[
\widetilde{r}(N)
\geq 
\mathfrak{S}(N) (1 - \mathcal{A}(N)) N
+ O \left( \frac{N}{(\log N)^{2+\frac{\epsilon}2}} \right).
\]
Since $\mathfrak{S}(N)\gg1$, it suffices to show that
\[
1 - \mathcal{A}(N) \gg \frac1{(\log N)^2}
\]
in order to prove $\widetilde{r}(N)>0$ for all sufficiently large even integers $N$.

Since the least prime not dividing $N$ is at most $2\log N$ when $N$ is large enough, we have
\[
1 - \mathcal{A}(N) 
 = 
1 - \prod_{\substack{p \nmid N \\ p > 2}}
\left( 1 - \frac{1}{p(p - 1)} \right)
> \frac{1}{2 \log N (2\log N - 1)},
\]
which confirms the above claim and therefore completes the proof of the corollary.

\end{document}